\documentclass[12pt]{amsart}
\usepackage{amsmath, amssymb, amsthm}
\usepackage{paralist, xcolor, tikz, tikz-cd, hyperref, mathtools, diagbox}

\usepackage{tikz}
\usetikzlibrary{patterns}
\usetikzlibrary{braids}

\usepackage[margin=1in,letterpaper,portrait]{geometry}

\allowdisplaybreaks

\usepackage{graphicx}
\usepackage{blkarray}
\usepackage{bigstrut}

\usepackage{caption}
\usepackage{enumitem}
\usepackage[noadjust]{cite}
\usepackage{comment}
\usepackage{tabularx}
\usepackage[normalem]{ulem}
\usepackage{calc}
\usepackage{subcaption}
\usepackage{bbm}

\newcommand{\WidestEntry}{$\cdots$}%
\newcommand{\SetToWidest}[1]{\makebox[\widthof{\WidestEntry}]{#1}}%

\renewcommand{\S}{\mathcal{S}}
\newcommand{\N}{\mathbb{N}}

\newcommand{\Z}{\mathbb{Z}}
\newcommand{\stat}{\sigma}   

\newcommand{\IS}{\mathsf{IS}}
\newcommand{\RW}{\mathsf{RW}}
\newcommand{\rw}{\mathsf{rw}}
\renewcommand{\S}{\mathcal{S}}      
\newcommand{\coeff}[2]{{#1}_{#2}}   
\newcommand{\set}{\mathrm{Set}}
\newcommand{\Eset}{\mathrm{ESet}}

\newcommand{\calib}[3]{[(#1,#2),i]}  

\newcommand\topstrut[1][1.2ex]{\setlength\bigstrutjot{#1}{\bigstrut[t]}}
\newcommand\botstrut[1][0.9ex]{\setlength\bigstrutjot{#1}{\bigstrut[b]}}

\theoremstyle{definition}
\newtheorem{theorem}{Theorem}[section]

\newtheorem{proposition}[theorem]{Proposition}
\newtheorem{lemma}[theorem]{Lemma}
\newtheorem{example}[theorem]{Example}

\newtheorem{corollary}[theorem]{Corollary}
\newtheorem{definition}[theorem]{Definition}

\title{Pattern expansions of permutation statistics}

\author[Ian Cavey]{Ian Cavey}
\address{Department of Mathematics, University of Illinois Urbana-Champaign, Urbana, IL, USA}
\email{cavey@illinois.edu}
\thanks{}

\author[Hugh Dennin]{Hugh Dennin$^*$}
\address{Department of Mathematics, The Ohio State University, Columbus, OH, USA}
\email{dennin.3@osu.edu}
\thanks{$^*$Research partially supported by NSF Grants DMS-2231565 and DMS-1945212.}

\author[Bridget Eileen Tenner]{Bridget Eileen Tenner$^\dagger$}
\address{Department of Mathematical Sciences, DePaul University, Chicago, IL, USA}
\email{bridget@math.depaul.edu}
\thanks{$^\dagger$Research partially supported by NSF Grant DMS-2054436.}

\date{}

\begin{document}

\begin{abstract}
    We study the expansions of permutation statistics in the basis of functions counting occurrences of a fixed pattern in a permutation. We show the finiteness of these pattern expansions for a class of permutation statistics including the higher moment statistics, generalizing a result of Berman and Tenner. We also give a combinatorial criterion for the positivity of pattern expansions. Using this criterion, we show that the pattern expansion of the number of reduced words of a permutation is positive and give an enumerative interpretation for the coefficients.
\end{abstract}

\maketitle

\section{Introduction}\label{section:intro}

A \textit{permutation statistic} is a function on $\S :=\bigsqcup_{n\geq 0} S_n$, the disjoint union of all finite permutation groups. Permutation statistics are ubiquitous, providing important metrics such as length, number of descents, and major index; see \cite{findstat} for an example of an important tool in this regard. A basic family of such statistics are the \textit{pattern-count functions} $[p]:\S\to \N$ indexed by permutations $p\in \S$ where for $w\in \S$, we write $[p](w)$ for the number of $p$-patterns in $w$ (see Definition~\ref{defn:pattern}).

Brändén and Claesson \cite{BrandenClaesson} observed that every permutation statistic has a unique expression as a (possibly infinite) linear combination of pattern-count functions (see also Section~\ref{section:tools}), which we call the \textit{pattern expansion} of the statistic. Moreover, the coefficients in this expansion are readily computable from the values of the statistic itself. 

In this paper, we develop the theory of pattern expansions of arbitrary permutation statistics.
Our perspective is that pattern expansions provide useful alternative representations of permutation statistics since they allow for the use of pattern-related tools and language to study a statistic of interest, even when that statistic had not initially appeared to relate to patterns.
For example, the expansion of a statistic could indicate that the statistic could have interesting properties when studied on particular classes of permutations, instead of on the full symmetric group, if many terms in the expansion are zero for such permutations.

We are particularly interested in statistics that are not, a priori, related to patterns, but whose pattern expansions have nice properties.
We give special attention to two appealing families of pattern expansions: those with only finitely many nonzero coefficients, and those with all nonnegative coefficients.
We will investigate the properties of permutations statistics in each of these two categories, and give non-trivial examples of each.

\subsection{Pattern-finite statistics}

We say that a permutation statistic is \textit{pattern-finite} if its pattern expansion has only finitely many nonzero coefficients. Pattern-finiteness appears to be a rare property among well-studied permutation statistics. Indeed, we observe in Proposition~\ref{prop: expected value} that the expected value of a pattern-finite statistic $\sigma$ on permutations of size $n$ is a polynomial function in $n$. This necessary condition can be used to show that many statistics of interest do not have finite pattern expansions.

Despite this, the following result due to Berman and Tenner \cite{BermanTenner} shows that some natural permutation statistics do exhibit surprising pattern-finiteness.

\begin{example}[\cite{BermanTenner}]\label{ex:variance} The \textit{variance} of a permutation, $V(w) = \sum_i (w(i)-i)^2$, has the finite pattern expansion:
$$V = 2\big([21] + [231] + [312] + [321]\big).$$
\end{example}

Motivated by this result, in Section~\ref{section:moments}, we give a generalization of this example, establishing the pattern-finiteness of a wide class of permutation statistics expressed in terms of marked permutations (see Definition~\ref{def: marked perm}). A special case of our Theorem~\ref{thm:averaging criterion} is that the higher moment statistics $\mu_m = \sum_i (w(i)-i)^m$
are pattern-finite for all $m\in \N$, and have a coefficient of zero on all pattern-count functions $[p]$ for $p\in S_k$ with $k>m+1$. Moreover, in Theorem~\ref{thm:moments are pattern-finite}, we give a formula for the coefficients in the expansion of $\mu_m$ in terms of Stirling numbers. For example, we find the following expansion for $\mu_3$ in Example~\ref{ex: moments}:
\[ \mu_3 = 6\big([312] - [231] + [4213]+[4123]+[4132]-[2431]-[2341]-[3241]\big). \]

For $m\geq 4$, we observe that certain linear combinations of the moments $\mu_1,\dots,\mu_m$, which we call the \textit{binomial moments} of a permutation, have simpler pattern expansions than the moment statistics themselves. Motivated by this observation, we investigate in Section~\ref{sec: higher moments} the properties of these binomial moment statistics, computing their pattern expansions and expected values. We hope that this serves as a case study for future such investigation and discovery of permutation statistics based on pattern expansions.

\subsection{Pattern-positive statistics}

We say that a permutation statistic is \textit{pattern-positive} if all coefficients in its pattern expansion are nonnegative. Every pattern-positive statistic $\stat$ is necessarily \textit{monotonic} with respect to pattern containment, meaning that $\stat(w)\ge \stat(p)$ whenever $w$ contains a $p$-pattern, since this inequality holds term-by-term in the pattern expansion of $\stat$. This necessary condition can be used to show that many statistics of interest are not pattern-positive.

A key property of any pattern-positive statistic $\stat$ is that one can obtain lower bounds for $\stat$ by truncating its pattern expansion. It is through this application that pattern expansions have appeared previously in the Schubert polynomial literature. Indeed, Weigandt \cite{Weigandt} showed that the Schubert polynomial $\mathfrak{S}_w$ satisfies $\mathfrak{S}_w(1,\dots,1)\geq 1+[132](w)$ and Gao \cite{Gao} strengthened this inequality by showing that $\mathfrak{S}_w(1,\dots,1)\geq 1+[132](w)+[1432](w)$. Gao then observed that both of these inequalities could be explained and further strengthened by the pattern-positivity of the statistic $\mathfrak{S}_w(1,\dots,1)$, the first several terms in the expansion being 
\[ \mathfrak{S}_w(1,\dots,1) = 1+[132]+[1432]+[12453]+[12534]+5[12543]+\cdots. \]
Several authors have since shown the nonnegativity of certain families of coefficients in this expansion \cite{Dennin,MeszarosTanjaya}, but the full pattern-positivity of $\mathfrak{S}_w(1,\dots,1)$ remains conjectural. 

In Section~\ref{section:reduced words} we show that pattern-positivity of any \emph{enumerative statistic} $\stat(w)=|\set(w)|$ follows from the existence of embeddings $\set(p)\hookrightarrow \set(w)$, indexed by pattern occurrences of $p$ in $w$, that satisfy a certain compatibility condition for intersections of patterns. Moreover, such a family of embeddings gives a positive combinatorial interpretation for the coefficients of the expansion, and in fact for every term in the computation of $\stat(w)$ from the pattern expansion of $\stat$.

Our main result, established using the strategy described above, is the pattern-positivity of $\rw(w)$, the number of reduced words for $w$ as a product of adjacent transpositions. The expansion of this statistic begins as follows:
\begin{align*}
    \rw  = 1 & + [321] + [2143] + [2413] + [2431] + [3142] + [3241] \\ 
    & + 2[3421] + [4132] + [4213] + 3[4231] + 2[4312] + 11[4321]+\cdots
\end{align*} 

As a consequence of this result, we obtain an infinite family of lower bounds for the number of reduced words $\rw(w)$ by truncating the pattern expansion of $\rw$. These are analogous to the bounds for $\mathfrak{S}_w(1,\dots,1)$ obtained by Weigandt and Gao. For example, keeping only the patterns in the expansion of $\rw$ of size at most $4$ we obtain the highly non-obvious inequality
\begin{align*}
    \rw  \ge  1 & + [321] + [2143] + [2413] + [2431] + [3142] + [3241] \\ 
    & + 2[3421] + [4132] + [4213] + 3[4231] + 2[4312] + 11[4321].
\end{align*}

Although the Schubert statistic $\mathfrak{S}_w(1,\dots,1)$ is closely related to reduced words, we have not been able to deduce the pattern-positivity of $\mathfrak{S}_w(1,\dots,1)$ from that of $\rw(w)$. Note, for example, that the vanishing of coefficients in the pattern expansion of either statistic does not imply the vanishing of the corresponding coefficient in the other. 

In Section~\ref{section:tools} of this work, we introduce the essential objects and properties that we will use in the paper. Section~\ref{section:moments} focuses on those pattern expansions that involve only finitely many nonzero coefficients, discussing properties of this class and presenting moment statistics as a relevant case study. With a similar framework, Section~\ref{section:reduced words} turns to those pattern expansions that involve only nonnegative coefficients. The case study for this important class of statistics is the number of reduced words of a permutation. We conclude the work in Section~\ref{section:future} with a survey of some of the many directions for further research on this topic.

\section{Essential tools}\label{section:tools}

Write $S_n$ for the set of permutations of $[1,n] = \{1,\dots,n\}$, and let $\S = \bigsqcup_{n\in \N} S_n$ denote the set of all permutations, where $S_0 = \{\varnothing\}$ contains only the empty permutation.
The \emph{size} of $w \in S_n$ is $|w| = n$. 
The building blocks of our work are permutation patterns and their occurrences. As such, we will represent permutations in one-line notation, writing $w \in S_n$ as $w(1)\cdots w(n)$. 

\begin{definition}\label{defn:pattern}
    Fix permutations $p \in S_k$ and $w \in S_n$. If $1\leq i_1 < \cdots < i_k\leq n$ are indices such that $w(i_1)\cdots w(i_k)$ is order isomorphic to $p(1)\cdots p(k)$ (i.e., the letters are in the same relative order), then $w$ contains a \emph{$p$-pattern} and $w(i_1)\cdots w(i_k)$ is a $p$-\emph{occurrence}, or an \emph{occurrence of $p$}, in $w$. 
    If $w$ has no occurrences of $p$, then $w$ \emph{avoids} $p$.
\end{definition}

We will sometimes want to look at the subword of a permutation appearing in particular positions, for which purpose we make the following definition.

\begin{definition}
    For $w\in S_n$ and $I = \{ i_1,\dots,i_h\}\subseteq [1,n]$ with $i_1<\cdots<i_h$, let $w|_I\in S_{h}$ be the permutation whose entries are in the same relative positions as the subword $w(i_1)\cdots w(i_h)$.
\end{definition}

A permutation that contains a $p$-pattern can contain many $p$-occurrences, and tracking that data is essential to our work.

\begin{definition}\label{defn:pattern-count}
    For any $p\in \S$, define the function
    $$[p]:\S\to \N$$
    sending a permutation $w$ to the number of $p$-occurrences in $w$. This $[p]$ is a \emph{pattern-count function}, or simply a \emph{pattern-count}.    
\end{definition}

Definition~\ref{defn:pattern-count} includes the pattern-count $[\varnothing]:\S\to\N$, indexed by the empty permutation $\varnothing \in S_0$. The function $[\varnothing]$ takes the value $1$ on all permutations. Similarly, for $1\in S_1$ we have $[1](w) = |w|$. For arbitrary $p\in S_k$ and $w\in S_n$, we have $[p](w) = 0$ if and only if $w$ avoids $p$.

Some well-known examples of permutation statistics are a permutation's size, length, number of descents, major index, and number of fixed points. We will typically consider permutation statistics such as these valued in $\N\subseteq \Z$, but our current discussion will make sense for statistics taking values in any (unital, commutative) ring $\Lambda$.

\begin{definition}\label{defn:permutation statistic}
    A \textit{$\Lambda$-valued permutation statistic} is any function $\stat\colon \S\to \Lambda$.
\end{definition}

Each pattern-count $[p]$ can be thought of as a permutation statistic taking values in $\Lambda$ by composing $[p]$ with the unique unital ring map $\Z \to \Lambda$.
We will study arbitrary permutation statistics by writing them as (possibly infinite) linear combinations of pattern-counts $[p]$. The following proposition asserts that this is well-defined.

\begin{proposition}[{cf.~\cite[\textsection 1.2]{BrandenClaesson}}]\label{prop:every statistic has a pattern expansion}
    Every permutation statistic $\stat\colon \S\to \Lambda$ can be expressed uniquely as a (possibly infinite) $\Lambda$-linear combination of pattern-count functions.
    That is, there is a unique collection of coefficients $\coeff{\stat}{p}\in \Lambda$ indexed by patterns $p\in \S$ such that
    $$
        \stat(w) = \sum_{p \in \S} \coeff{\stat}{p} \cdot [p](w)
    $$
    for all permutations $w\in \S$. 
\end{proposition}

When $\stat$ and $\coeff{\stat}{p}$ are related as in Proposition~\ref{prop:every statistic has a pattern expansion}, we write $\stat = \sum_{p\in \S} \coeff{\stat}{p} \cdot [p]$ and call this expression the \textit{pattern expansion} of the statistic $\stat$.
For a fixed $w\in \S$, the value $[p](w)$ is nonzero for only finitely many $p\in \S$ since $[p](w) = 0$ if $|p| > |w|$, so evaluating a pattern expansion at any fixed $w$ makes sense.

Although proved elsewhere, we include our own proof of Proposition~\ref{prop:every statistic has a pattern expansion} here because it sets up a perspective that we find helpful in later discussions.

\begin{proof}[Proof of Proposition~\ref{prop:every statistic has a pattern expansion}]
    Let $M$ be the infinite matrix with columns indexed by $p\in \S$ and rows indexed by $w\in \S$, with corresponding entries equal to $[p](w)$. Order the rows and columns of $M$ by any order on $\S$ that refines the grading by permutation size $S_0,S_1,S_2,\dots$. For any $w \in S_n$ and $|p| \ge n$, the value $[p](w)$ is equal to $0$ unless $p=w$, in which case $[p](w)=1$. This means that $M$ is lower-triangular with $1$s along the main diagonal. In particular, the matrix $M$ is invertible and the entries of $M^{-1}$ are integers.

    Given a statistic $\stat\colon \S\to \Lambda$,
    let $s$ denote the column vector with entries $\stat(w)$ indexed by $w\in \S$ in the same order as in $M$.
    By the previous observations about $M$, there is a unique solution $c$ to the matrix equation $Mc=s$ whose entries are elements of $\Lambda$ indexed by $\S$.
    Letting $\coeff{\stat}{p}$ denote the entry of $c$ indexed by $p\in \S$, the matrix equation $Mc=s$ is equivalent to the expression
    $\stat=\sum_{p\in \S} \coeff{\stat}{p} \cdot [p].$
\end{proof}

Throughout this work, we will maintain the notation used above to represent coefficients in pattern expansions.

\begin{definition}
    Given a permutation statistic $\stat\colon \S \to \Lambda$, let $\coeff{\stat}{p}$ be the coefficient of $[p]$ in the pattern expansion of $\stat$.
\end{definition}

The matrix $M$ appearing in the proof of Proposition~\ref{prop:every statistic has a pattern expansion} can be thought of as the change of basis matrix between the basis of pattern-count functions $[p]$ and the basis of indicator functions $\mathbbm{1}_p$ defined by 
$$\mathbbm{1}_p(w) = \begin{cases}
    1 & \text{if } p=w, \text{ and}\\
    0 & \text{otherwise}.
\end{cases}$$
This description is useful computationally, as the restriction of the matrices $M$ and $M^{-1}$ to permutations of size at most $n$ can be computed independently of any chosen statistic. Then, the coefficients $\coeff{\stat}{p}$ on patterns $p$ of size at most $n$ in the pattern expansion of any statistic $\stat$ can be computed directly from these matrices and the vector of outputs of $\stat$ on permutations.

\begin{example}
    The matrix equation $Mc=s$, referred to in the proof of Proposition~\ref{prop:every statistic has a pattern expansion}, for the statistic \textit{number of descents} is as follows.
$$\begin{blockarray}{ccccccccccclccc}
    & \scriptstyle{[\varnothing]} & \scriptstyle{[1]} & \scriptstyle{[12]} & \scriptstyle{[21]} & \scriptstyle{[123]} & \scriptstyle{[132]} & \scriptstyle{[213]} & \scriptstyle{[231]} & \scriptstyle{[312]} & \scriptstyle{[321]} & \cdots\\
    \\[-2ex]
    \begin{block}{r[ccccccccccl][r]c[c]}
    \scriptstyle{\varnothing\ }\  & \SetToWidest{1} & \SetToWidest{0} & \SetToWidest{0} & \SetToWidest{0} & \SetToWidest{0} & \SetToWidest{0} & \SetToWidest{0} & \SetToWidest{0} & \SetToWidest{0} & \SetToWidest{0} & \cdots & 0\ \ && \ 0 \ \topstrut \\
    \scriptstyle{1\ } & 1 & 1 & 0 & 0 & 0 & 0 & 0 & 0 & 0 & 0 & \cdots & 0\ \ && 0 \\
    \scriptstyle{12\ } & 1 & 2 & 1 & 0 & 0 & 0 & 0 & 0 & 0 & 0 & \cdots & 0 \ \ && 0 \\
    \scriptstyle{21\ } & 1 & 2 & 0 & 1 & 0 & 0 & 0 & 0 & 0 & 0 & \cdots & 1 \ \ && 1  \\
    \scriptstyle{123\ } & 1 & 3 & 3 & 0 & 1 & 0 & 0 & 0 & 0 & 0 & \cdots & 0 \ \ && 0 \\
    \scriptstyle{132\ } & 1 & 3 & 2 & 1 & 0 & 1 & 0 & 0 & 0 & 0 & \cdots & 0 \ \ & = & 1 \\
    \scriptstyle{213\ } \ & 1 & 3 & 2 & 1 & 0 & 0 & 1 & 0 & 0 & 0 & \cdots & 0\ \ && 1 \\
    \scriptstyle{231\ } & 1 & 3 & 1 & 2 & 0 & 0 & 0 & 1 & 0 & 0 & \cdots & \,-1\ \ && 1 \\
    \scriptstyle{312\ } & 1 & 3 & 1 & 2 & 0 & 0 & 0 & 0 & 1 & 0 & \cdots & \,-1\ \ && 1 \\
    \scriptstyle{321\ } & 1 & 3 & 0 & 3 & 0 & 0 & 0 & 0 & 0 & 1 & \cdots & \,-1\ \ && 2 \\
    \vdots\phantom{\ \ } & \vdots & \vdots & \vdots & \vdots & \vdots & \vdots & \vdots & \vdots & \vdots & \vdots & \ddots  & \vdots \ \ \, && \vdots \botstrut\\
    \end{block}
\end{blockarray}\vspace*{-1.25\baselineskip}.$$

\noindent Row- and column-indexing in this matrix is identified for patterns of size at most $3$, with rows indexing the function input $w$ and columns indexing the terms $[p]$ in the pattern expansion. Note that the blocks of the matrix corresponding to $[p](w)$ for $p$ and $w$ of the same size contain identity matrices, and the entire matrix is lower-triangular. From this matrix equation we obtain the expansion for the number of descents of a permutation: 
\[ \mbox{des}(w) = \big([21]-[231]-[312]-[321]+\cdots\big)(w).\]
\end{example}

The entries of the matrix $M^{-1}$ can be given explicitly, and the rearranged matrix equation $c = M^{-1}s$ gives the following useful expression for the coefficients $\coeff{\stat}{p}$ of a statistic $\stat$ directly in terms its values.

\begin{lemma}\label{lem:coeff in terms of stat}
    For any permutation statistic $\stat$ and $w\in S_n$, the coefficient $\coeff{\stat}{w}$ is given by
    \begin{align*}
        \coeff{\stat}{w} & = \sum_{I\subseteq [1,n]} (-1)^{n-|I|}\cdot\stat(w|_I)\\
        & = \sum_{p \in \S} (-1)^{n-|p|} \cdot [p](w)\cdot\stat(p).
    \end{align*}
\end{lemma}

Note that the later sum is finite because $[w](p) = 0$ if $w\in S_n$ with $n>k$.

\begin{proof}
    When applied to the permutation $w$, the pattern expansion formula for $\stat$ given in Proposition~\ref{prop:every statistic has a pattern expansion}, which defines the coefficients $\coeff{\stat}{p}$, can now be rewritten as 
    \begin{equation*}
        \stat(w) = \sum_{p \in \S} \ \sum_{\substack{I \subseteq [1,n] \\ w|_I = p}} \coeff{\stat}{p} 
        = \sum_{I\subseteq [1,n]} \coeff{\stat}{w|_I}.
    \end{equation*} 
    By an inclusion-exclusion type argument, one sees that these equations are satisfied by
    \begin{equation*}
    \coeff{\stat}{w} = \sum_{I\subseteq [1,n]} (-1)^{n-|I|}\cdot\stat(w|_I). 
    \end{equation*}
    The second expression is obtained by grouping together all subsets $I\subseteq [1,n]$ for which $w|_I$ is some fixed pattern $p$, and writing the sum over all possible $p$. The number $I\subseteq [1,n]$ for which $w|_I = p$ is exactly $[p](w)$, which completes the proof.
\end{proof}

We note that while Definitions~\ref{defn:pattern-count} and~\ref{defn:permutation statistic} and Proposition~\ref{prop:every statistic has a pattern expansion} are phrased in the language of classical pattern containment, they could just as easily be written for any version of pattern containment.
For example, an analogous argument to that presented in the proof of Proposition~\ref{prop:every statistic has a pattern expansion} would show that every statistic can be expressed uniquely as an integer-weight sum of consecutive-pattern-counts, as well. The present work aims to study and build a foundation of classical pattern expansions, and we leave analogous efforts for other versions to other works and, perhaps, authors.

\subsection{The algebras of permutation statistics and pattern expansions}

We close this section with a brief detour to discuss algebras associated to the objects we are studying in this work.

Write $\Lambda^\S$ for the collection of permutation statistics valued in $\Lambda$.
Under the operations of pointwise addition and multiplication, this $\Lambda^\S$ becomes a commutative algebra over $\Lambda$ with unit $[\varnothing]$ (the constant function $1$).

One special class of permutation statistics is the collection of pattern-finite statistics: those statistics whose pattern expansion has finitely many nonzero terms.
The collection of pattern-finite statistics $\Lambda\S$ forms a $\Lambda$-submodule of $\Lambda^\S$.
In fact, more is true.

\begin{proposition}[{\cite{Vargas}}]\label{prop:free algebra}
    The collection of pattern-finite statistics $\Lambda\S$ is a $\Lambda$-subalgebra of $\Lambda^\S$. 
\end{proposition}

It follows from Proposition~\ref{prop:free algebra} that the product of two pattern-count functions $[p]$ and $[q]$ is a \textit{finite} linear combination of other pattern-count functions $[r]$; that is, it is a statistic with a pattern-finite expansion. In particular, Vargas showed for $p,q\in \S$ that \[
    [p][q] = \sum_{r\in \S} d^{pq}_r \cdot [r]
\]
where $d^{pq}_r\in \N$ denotes the number of ways to write $[1,|r|] = I\cup J$ as the union of two (not necessarily disjoint) subsets such that $I$ indexes a $p$-occurrence in $r$ and $J$ indexes a $q$-occurrence in $r$.
Notice that $d^{pq}_r = 0$ if $|p| + |q| < |r|$, so this sum is finite.

\section{Pattern-finite expansions}\label{section:moments}

In a sense, the optimal family of statistics to study via pattern-count functions are those with pattern-finite expansions.
Having a finite pattern expansion gives the statistic a direct combinatorial pattern-related meaning.
An interesting question which several authors have studied (for instance, see \cite{BabsonSteingrimsson,BermanTenner}) is whether or not a given permutation statistic of interest is pattern-finite.

In this section, we give a necessary criterion for a permutation statistic to be pattern-finite in terms of its expected value. We also establish the pattern-finiteness of a large category of permutation statistics expressed in terms of marked permutations. As a case study, we show that the higher moment statistics are pattern-finite, generalizing a result of~\cite{BermanTenner}, and study their pattern expansions.

We begin with a necessary criterion for pattern-finiteness.

\begin{proposition}\label{prop: expected value}
    For $\sigma$ a finite linear combination of pattern-counts $[p]$ with all patterns of size at most $d$, the expected value of $\sigma$ on permutations in $S_n$ is a polynomial function in $n$ of degree at most $d$.
\end{proposition}

\begin{proof}
     This follows from linearity of the expected value and the observation that for any $p\in S_k$, the expected value of the pattern-count function $[p]$ on $S_n$ is $\frac{1}{k!}\binom{n}{k}$, a polynomial in $n$ of degree $k$.
\end{proof}

This necessary condition for pattern-finiteness is sufficient to show that many statistics of interest are not pattern-finite, even using very weak bounds for the expected value. 

\begin{example}
    The statistic $\rw(w)$, the number of reduced words of $w$, is not pattern-finite. Indeed, Stanley \cite{Stanley} showed that the number of reduced words of the longest word $w_0\in S_n$ is 
    \[ \frac{\binom{n}{2}!}{1^{n-1}3^{n-2}5^{n-3}\cdots (2n-3)^1}. \]
    This quantity divided by $n!$ is therefore a lower bound 
    for the expected number of reduced words for a permutation in $S_n$, and one can show that this lower bound has faster than exponential growth in $n$.
\end{example}

Even if the expected value of a statistic on $S_n$ is polynomial in $n$, the statistic can fail to be pattern-finite. For example, the expected \textit{number of descents} of a permutation in $S_n$ is $(n-1)/2$, but this statistic has an infinite pattern expansion~\cite{BrandenClaesson}.

We now turn to the pattern-finiteness of a family of permutation statistics related to the \emph{calibrated} patterns of \cite{Tenner-prisms} which refine the usual pattern-count functions $[p]$.
As a prototypical example, recall that the \textit{code} of a permutation $w$ is the sequence $(c_1(w),c_2(w),\dots)$ where $c_i(w) := \#\{j>i \mid w(i)>w(j)\}$.
This refines the length of $w$ in that $\ell(w) = [21](w) = \sum_{i\geq 1} c_i(w)$, which follows by regarding each $c_i(w)$ as enumerating the occurrences $w(i)w(j)$ of $21$ in $w$.
We generalize these ideas to other pattern-count functions $[p]$ using the language of marked permutations.

\begin{definition}\label{def: marked perm}
A \textit{marked permutation} is a pair $(p,h)$ where $p\in \S$ and $1\leq h\leq |p|$.
\end{definition}

Write $\S'$ for the set of marked permutations.
We will often notate marked permutations $(p,h)\in \S'$ by writing $p$ in one-line notation and underlining the value at position $h$.
For instance, by $2\underline{3}1$ we mean the marked permutation $(231,2)$.

\begin{definition}
Fix $i\geq 1$ and $(p,h)\in \S'$.
\begin{itemize}
    \item
    An \textit{$i$-calibrated} occurrence of $(p,h)$ in $w\in \S$ is an occurrence $w(i_1) \cdots w(i_{|p|})$ of $p$ in $w$ satisfying $i_h = i$.
    \item
    Let $\calib{p}{h}{i}\colon \S \rightarrow \mathbb{N}$ denote the permutation statistic sending each $w\in \S$ to the number of $i$-calibrated occurrences of $(p,h)$ in $w$.
    Call $\calib{p}{h}{i}$ an \emph{$i$-calibrated pattern-count function}, or simply an \emph{$i$-pattern-count}.
\end{itemize}
\end{definition}

With this notation, permutation codes can be written as $c_i = [\underline{2}1,i]$, so our refinement of the length statistic becomes $[21] = \sum_{i\geq 1} [\underline{2}1,i]$.
This observation generalizes in the obvious way: for any $p\in\S$, each $1\leq h\leq |p|$ gives a decomposition $[p] = \sum_{i\geq1} \calib{p}{h}{i}$.

As a further example, we have $[2\underline{3}1,4](153642) = 3$ because $362, 562, 564$ are the only $231$-occurrences $w(i_1)w(i_2)w(i_3)$ in $w$ satisfying $i_2 = 4$.

Say that a permutation statistic is \emph{$i$-pattern-finite} if it can be written as a finite linear combination of $i$-pattern-count functions.
Write $\Lambda\S'_i$ for the subspace of $\Lambda^\S$ consisting of $i$-pattern-finite statistics.
The next proposition shows that each $\Lambda\S'_i$ is a subalgebra of $\Lambda^\S$ (cf.~Proposition~\ref{prop:free algebra}) and, moreover, that products can be computed independent of $i$.

\begin{proposition}
    \label{prop:i-pattern products}
    There exist nonnegative coefficients $d_*^{**}\in \Z_{\geq0}$, triply indexed by $\S'$, satisfying the following:
    \begin{enumerate}
    \item[(i)]
    $d_{(r,h_3)}^{(p,h_1)(q,h_2)} = 0$ whenever $|p| + |q| \leq |r|$.
    
    \item[(ii)]
    For any $i\geq 1$, the identity
    \begin{equation}
        \label{eq:i-pattern products}
        \calib{p}{h_1}{i} \cdot \calib{q}{h_2}{i} = \sum_{\substack{(r,h_3) \in \S'}} d_{(r,h_3)}^{(p,h_1)(q,h_2)} \cdot \calib{r}{h_3}{i}
    \end{equation}
    holds in $\Lambda\S'_i$. 
    In particular, each $\Lambda\S'_i$ is a subalgebra of $\Lambda^\S$.
    \end{enumerate}
\end{proposition}

\begin{proof}
    Let $d_{(r,h_3)}^{(p,h_1)(q,h_2)}$ denote the number of ways to write $[1,|r|]$ as a (not necessarily disjoint) union $A\cup B$, where $A = \{a_1<\cdots<a_{|p|}\}$ and $B = \{b_1<\cdots<b_{|q|}\}$ index $h_3$-calibrated occurrences in $r$ of $(p,h_1)$ and $(q,h_2)$, respectively.
    It is clear that (i) is satisfied.

    We show (ii) holds by evaluating both sides of Equation~\eqref{eq:i-pattern products} on a permutation $w\in \S$.
    The left-hand side counts pairs $(I,J)$ where $I = \{i_1<\cdots<i_{|p|}\}$ and $J = \{j_1<\cdots<j_{|q|}\}$ index $i$-calibrated occurrences in $w$ of $(p,h_1)$ and $(q,h_2)$, respectively.
    The right-hand side counts triples $(K,A,B)$ where $K = \{k_1<\cdots<k_{|r|}\}$ indexes an $i$-calibrated occurrence in $w$ of some $(r,h_3)\in \S'$ and $[1,|r|] = A\cup B$ is a decomposition counted by $d_{(r,h_3)}^{(p,h_1)(q,h_2)}$.
    We will demonstrate that these sets are in bijection.
    
    To go from $(I,J)$ to $(K,A,B)$, set $K := I\cup J$ (which necessarily indexes an $i$-calibrated occurrence in $w$ of some $(r,h_3)\in \S'$) and set $A := \{a \mid k_a\in I\}$ and $B := \{b \mid k_b\in J\}$.
    Then $A$ indexes a $p$-occurrence in $r$, since $r(a_1)\cdots r(a_{|p|})$ has the same relative order as $w(k_{a_1})\cdots w(k_{a_{|p|}}) = w(i_1)\cdots w(i_{|p|})$.
    Furthermore, $k_{a_{h_1}} = i_{h_1} = i = k_{h_3}$ implies $a_{h_1} = h_3$, so this is an $h_3$-calibrated occurrence of $(p,h_1)$.
    A similar argument shows that $B$ indexes an $h_3$-calibrated occurrence of $(q,h_2)$ in $r$.

    Conversely, starting with $(K,A,B)$, we can take $I := \{k_a \mid a\in A\}$ and $J := \{k_b \mid b\in B\}$.
    The set $I$ indicates an occurrence of $p$ in $w$, since $w(i_1)\cdots w(i_{|p|}) = w(k_{a_1})\cdots w(k_{a_{|p|}})$ has the same relative order as $r(a_1)\cdots r(a_{|p|})$.
    We also know that $i_{h_1} = k_{a_{h_1}} = k_{h_3} = i$, so this is an $i$-calibrated occurrence of $(p,h_1)$.
    Likewise, the set $J$ indicates a $i$-calibrated occurrence of $(q,h_2)$ in $w$.
    These constructions are inverse to one another, yielding the desired bijection.
\end{proof}

We can now state our criterion for pattern-finiteness.

\begin{theorem}\label{thm:averaging criterion}
    Fix $s$ marked permutations $(p_1,h_1), \ldots, (p_s,h_s) \in \S'$, and a polynomial \mbox{$f(x_1, \ldots, x_s) \in \Lambda[x_1,\ldots, x_s]$.}
    The permutation statistic
    \begin{equation}
        \label{eq:averaging criterion}
        \sum_{i=1}^\infty f\big(\calib{p_1}{h_1}{i}, \ldots, \calib{p_s}{h_s}{i}\big)
    \end{equation}
    is pattern-finite.
\end{theorem}

\begin{proof}
    By Proposition~\ref{prop:i-pattern products}, we can expand each term in Equation~\eqref{eq:averaging criterion} as a linear combination \[
        f\big(\calib{p_1}{h_1}{i},\dots,\calib{p_s}{h_s}{i}\big)
        =\sum_{(r,h)\in \S'} C_{(r,h)} \cdot \calib{r}{h}{i}
    \] where the coefficients $C_{(r,h)}\in \Lambda$ are independent of $i$, and only finitely many are nonzero.
    Averaging, then, gives \begin{align*}
        \sum_{i=1}^\infty f\big(
            \calib{p_1}{h_1}{i},\dots,\calib{p_s}{h_s}{i}
        \big)
        &= \sum_{i=1}^\infty \sum_{(r,h)\in \S^*} C_{(r,h)} \cdot \calib{r}{h}{i}\\
        &= \sum_{(r,h)\in \S'} C_{(r,h)} \left(\sum_{i=1}^\infty \calib{r}{h}{i}\right)\\
        &= \sum_{(r,h)\in \S'} C_{(r,h)} \cdot [r]\\
        &= \sum_{r\in \S} \left(\sum_{h=1}^{|r|} C_{(r,h)}\right) \cdot [r],
    \end{align*}
    so the statistic has a finite pattern expansion.
\end{proof}

\subsection{A case study: moment statistics}
\label{sec: higher moments}

For $i\geq 1$, write $\Delta_i$ for the $i$th difference statistic given by \[
    \Delta_i(w) := w(i) - i,
\] where by convention $w(i) = i$ whenever $i>|w|$.
Notice that the permutation variance can be expressed as $V = \sum_{i=1}^\infty \Delta_i^2$.
More generally, given $m\geq 1$, we can consider the \emph{$m$th permutation moment} \[
    \mu_m(w)
    := \sum_{i=1}^\infty \Delta_i(w)^m
    = \sum_{i=1}^\infty (w(i) - i)^m.
\]

Recall that $V$ is pattern-finite \cite{BermanTenner}.
Indeed, using Theorem\ref{thm:averaging criterion}, we will see that the same holds for higher permutation moments as well.
The key observation is the following:

\begin{lemma}
\label{lem:Delta_i expansion}
For $i\geq 1$, we have $\Delta_i = [\underline{2}1,i] - [2\underline{1},i]$.
\end{lemma}

\begin{proof}
Observe that \begin{align*}
    \Delta_i(w) &= w(i) - i\\
    &= \#\{ j \, | \, w(j)< w(i) \} - \# \{ j \, | \, j<i \}\\
    &= \# \{ j \, | \, j>i \text{ and } w(j)<w(i)\} - \# \{ j \, | \, j<i \text{ and } w(j)>w(i)\}\\
    &= [\underline{2}1,i](w) - [2\underline{1},i](w).\qedhere
\end{align*}
\end{proof}

To state our more precise results on the pattern expansion of $\mu_m$, we need the following notation:
for $p\in S_a$ and $q\in S_b$, let $p\ominus q\in S_{a+b}$ denote the skew sum of $p$ and $q$ given by
\[
    (p\ominus q)(i) = \begin{cases}
        b+p(i) & \text{if } 1\leq i\leq a, \text{ and} \\
        q(i-a) & \text{if } a+1 \leq i\leq a+b.
    \end{cases}
\]

\begin{theorem}\label{thm:moments are pattern-finite}
For $m\geq 1$, the $m$th permutation moment $\mu_m$ is pattern-finite.
Explicitly, \[
    \mu_m
    = \sum_{\substack{p,q\in \S \\ |p|+|q|\leq m}} \kappa(m,|p|,|q|) \cdot [q\ominus 1\ominus p]
\]
where \[
    \kappa(m,a,b) 
    := \sum_{k=b}^{m-a} (-1)^k \binom{m}{k} \,a!\, S(m-k,a) \,b!\, S(k,b).
\]
Here $S(n,k)$ denotes a Stirling number of the second kind, which counts the number of set partitions of $[1,n]$ into $k$ nonempty blocks.
Equivalently, $k! S(n,k)$ counts the number of surjections from $[1,n]$ to $[1,k]$.
\end{theorem}

\begin{proof}
We can express $\mu_m = \sum_{i\geq1} \left([\underline{2}1,i] - [2\underline{1},i]\right)^m$ using Lemma~\ref{lem:Delta_i expansion},
so $\mu_m$ is automatically pattern-finite by Theorem~\ref{thm:averaging criterion}.
To compute the pattern expansion of $\mu_m$, we expand: \[
    \Delta_i^m
    = \left([\underline{2}1,i] - [2\underline{1},i]\right)^m
    = \sum_{k=0}^m (-1)^k \binom{m}{k} \cdot [\underline{2}1,i]^{m-k} \cdot [2\underline{1},i]^k.
\]
The statistic $[\underline{2}1,i]^{m-k} [2\underline{1},i]^k$ evaluated at $w$ counts the number of ways to choose indices $i'_1,\dots,i'_k< i < i''_1,\dots,i''_{m-k}$ such that $w(i'_j) > w(i)$ for all $1\leq j\leq k$ and $w(i) > w(i''_j)$ for all $1\leq j\leq m-k$.
The subword of $w$ on these indices will be an occurrence of some permutation $q\ominus 1\ominus p$ where $|p|\leq m-k$ and $|q|\leq k$, and every such occurrence arises this way in $|p|!\,S(m-k,|p|)\,|q|!\,S(k,|q|)$ ways.
Hence \begin{align*}
    \Delta_i^m
    &= \sum_{k=0}^m (-1)^k \binom{m}{k} \left(
    \sum_{\substack{p,q\in \S \\ |p|\leq m-k \\ |q|\leq k}}
        |p|! \,|q|! \,S(m-k,|p|)\, S(k,|q|) \cdot \calib{q\ominus 1\ominus p}{|q|+1}{i}
    \right) \\
    &= \sum_{\substack{p,q\in \S \\ |p|+|q|\leq m}}
    \underbrace{\left(
         \sum_{k=|q|}^{m-|p|} (-1)^k \binom{m}{k} |p|! \,S(m-k,|p|)\, |q|!\, S(k,|q|)
    \right)}_{\kappa(m,|p|,|q|)} \cdot \calib{q\ominus 1\ominus p}{|q|+1}{i}.
\end{align*}
The result follows from summing over all $i\geq1$.
\end{proof}

\begin{example}\label{ex: moments}
Recall that 
$$
    \mu_1 = 0 \hspace{.25in} \text{and} \hspace{.25in}
    \mu_2 = 2\big([21] + [231] + [312] + [321]\big).
$$
We will now compute the pattern expansion of $\mu_3$.
Note that the coefficients $\kappa(m,a,b)$ satisfy the following properties:
\begin{itemize}
\item $\kappa(m,a,b) = (-1)^m\kappa(m,b,a)$. In particular, $\kappa(m,a,a) = 0$ whenever $m$ is odd.
\item $\kappa(m,a,0) = a! \, S(m,a) = (-1)^m\kappa(m,0,a)$.
\item $\kappa(a+b,a,b) = (-1)^b (a+b)!$.
\end{itemize}
The coefficients $\kappa(3,a,b)$ are thus given by \[
    \begin{tabular}{|c||c|c|c|c|}
        \hline
        \backslashbox{$a$}{$b$} & $\phantom{-}0\phantom{-}$ \ & $1$ & $2$ & $3$ \\ \hline\hline
        $0$ & $0$ & $-1\phantom{-}$ & $-6\phantom{-}$ & $-6\phantom{-}$ \\ \hline
        $1$ & $1$ & $0$  & $6$  & $0$  \\ \hline
        $2$ & $6$ & $-6\phantom{-}$ & $0$  & $0$  \\ \hline
        $3$ & $6$ & $0$  & $0$  & $0$  \\ \hline
\end{tabular}
\]
so we get the pattern expansion
\begin{align*}
    \mu_3
    &= \big([21] - [21]\big) + 6\Big(\sum_{p\in S_2} [1\ominus p] - \sum_{q\in S_2} [q\ominus 1]\Big) \\
    &\qquad + 6\Big(\sum_{p\in S_3} [1\ominus p] - \sum_{\substack{p\in S_2}} [21\ominus p] + \sum_{\substack{q\in S_2}} [q\ominus 21] - \sum_{q\in S_3} [q\ominus 1] \Big) \\
    &= 6\big([312] - [231]
    + [4213]+[4123]+[4132]-[2431]-[2341]-[3241]\big).
\end{align*}
\end{example}

Despite the previous examples, it is not always the case that $\mu_m$ is divisible by $m!$.
For example, we have $\mu_4(231) = (2-1)^4 + (3-2)^4 + (1-3)^4 = 1 + 1 + 16 = 18$.
However, it turns out to be the case that
\begin{align*}
    \mu_4
    &= 24\big([312] + [321] + [2341] + [2431] + [3241] + [3421] \\
    &\quad + 2[4123] + 2[4132] + 2[4213] + [4312] + 2[4321] + 3[4231] \\
    &\quad + [23451] + [23541] + [24351] + [24531] + [25341] + [25431] + [32451] + [32541] \\ 
    &\quad + [34251] + [35241] + [42351] + [42531] + 
    [43251] + [45231] + [45312] + [51234] \\
    &\quad + [51243] + [51324] + [51342] + [51423] + [51432] + [52134] + [52143] + [52314] \\
    &\quad + [52413] + [53124] + [53142] + [53214] + [53412] + [53421] + [54231] + [54312] \\
    &\quad + [54321] + 2[45321] + 2[52341] + 2[52431] + 2[53241]\big) + 2\mu_3 + \mu_2.
\end{align*}
One explanation for this divisibility is that there is a factorization (recalling that $\mu_1 = 0$) \begin{align*}
    \mu_4 - 2\mu_3 - \mu_2 + 2\mu_1
    &= \sum_{i\geq 1}
    (\Delta_i^4 - 2\Delta_i^3 - \Delta_i^2 + 2\Delta_i) \\
    &= \sum_{i\geq 1}(\Delta_i + 1)(\Delta_i)(\Delta_i - 1)(\Delta_i - 2).
\end{align*}
This suggests that the following variant of the permutation moment statistic is more natural.

\begin{definition}
    Define the \emph{$m$th binomial permutation moment} to be the statistic 
    \[
    \overline{\mu}_m(w)
    := \sum_{i\geq1} \binom{\Delta_i(w) + \lceil m/2\rceil - 1}{m}
    = \sum_{i\geq1} \binom{w(i) - i + \lceil m/2\rceil - 1}{m}.
    \]
\end{definition}

As suggested by the discussion above, we find that the coefficients in pattern expansion of $\overline{\mu}_m$ are given by the following simpler expressions than those for $\mu_m$.

\begin{proposition}
For $m\geq 1$, the $m$th binomial permutation moment $\overline{\mu}_m$ is pattern-finite, expanding as
\[
    \overline{\mu}_m = \sum_{\substack{p,q\in \S \\ |p|+|q|\leq m}} \overline{\kappa}(m,|p|,|q|) \cdot [q\ominus 1\ominus p],
\]
where \[
    \overline{\kappa}(m,a,b) 
    := \sum_{k= b}^{m-a} (-1)^k \binom{\lceil m/2\rceil-1}{m-a-k} \binom{k-1}{k-b}.
\]
\end{proposition}

\begin{proof}
The coefficients $\overline{\kappa}(m,a,b)$ are defined so that the identity \begin{equation}
\label{eq:kappa bar identity}
    \binom{\Delta_i+\lceil m/2\rceil-1}{m}
    = \sum_{\substack{a,b\geq 0}} \overline{\kappa}(m,a,b) 
    \binom{[\underline{2}1,i]}{a}
    \binom{[2\underline{1},i]}{b}
\end{equation}
holds.
To see this, one needs the following well-known identity of binomial coefficients:
\begin{equation}
\label{eq:a+b choose c identity}
    \binom{\alpha+\beta}{\gamma} = \sum_{a\geq 0} \binom{\alpha}{a}\binom{\beta}{\gamma-a}.
\end{equation}
A first computation shows that for $x,y,m\geq 0$, there is an identity
\begin{align*}
    \binom{x-y}m
    &= \sum_{a\geq 0} \binom{x}{a}\binom{-y}{m-a} \\
    &=\sum_{a\geq 0} (-1)^{m-a}
    \binom{x}{a}\binom{y+m-a-1}{m-a} \\
    &= \sum_{a,b\geq 0}
    (-1)^{m-a} \binom{m-a-1}{m-a-b}
    \binom{x}{a}\binom{y}{b}
\end{align*}
where the first and last equalities use Equation~\ref{eq:a+b choose c identity}
and the second equality uses the identity $(-1)^\beta\binom{-\alpha}{\beta} = \binom{\alpha+\beta-1}{\beta}$.
More generally, for $x,y,m,r\geq 0$ we have that \begin{align*}
    \binom{x-y+r}{m}
    &= \sum_{k\geq 0} \binom{r}{m-k} \binom{x-y}{k} \\
    &= \sum_{k\geq 0} \binom{r}{m-k}
    \left(
        \sum_{a,b\geq 0}
        (-1)^{k-a} \binom{k-a-1}{k-a-b}
        \binom{x}{a}\binom{y}{b}
    \right) \\
    &= \sum_{a,b\geq 0}
    \left(
        \sum_{k=a+b}^m (-1)^{k-a} \binom{r}{m-k} \binom{k-a-1}{k-a-b}
    \right)
    \binom{x}{a}\binom{y}{b}\\
    &= \sum_{a,b\geq 0}
    \left(
        \sum_{\ell=b}^{m-a} (-1)^\ell \binom{r}{m-a-\ell} \binom{\ell-1}{\ell-b}
    \right)
    \binom{x}{a}\binom{y}{b},
\end{align*}
where the first line is a final application of Equation~\ref{eq:a+b choose c identity},
the second line is the identity from our previous computation, and the last equality is a re-indexing of the inner sum with $\ell = k-a$.
Equation~\ref{eq:kappa bar identity} now follows by taking $x = [\underline{2}1,i]$, $y = [2\underline{1},i]$, and $r = \lceil m/2 \rceil - 1$, recalling that $\Delta_i = [\underline{2}1,i] - [2\underline{1},i] = x - y$.

Notice that the statistic $\binom{[\underline{2}1,i]}{a}\binom{[2\underline{1},i]}{b}$
evaluated at $w\in \S$ counts the ways to pick indices $i'_1 < \dots < i'_b < i < i''_1 < \dots < i''_a$ such that $w(i'_j) > w(i)$ for all $1\leq j\leq b$ and $w(i) > w(i''_j)$ for all $1\leq j\leq a$.
This is the same data as an occurrence of a pattern $q\ominus 1\ominus p$, where $p\in S_a$ and $q\in S_b$, hence \[
    \binom{[\underline{2}1,i]}{a} \binom{[2\underline{1},i]}{b}
    = \sum_{\substack{p\in S_a \\ q\in S_b}}
    [(q\ominus 1\ominus p,b+1),i].
\]
Combining this with Equation~\ref{eq:kappa bar identity}, we conclude that \[
    \binom{\Delta_i+\lceil m/2\rceil-1}{m}
    = \sum_{p,q\in \S} \overline{\kappa}(m,|p|,|q|) \cdot [(q\ominus 1\ominus p,|q|+1),i].
    \qedhere
\]
\end{proof}

Our discovery of the binomial permutation moment statistics was aided by the following simple expression for their expected values, which we now establish directly.

\begin{proposition}
\label{prop:binom moment expected value}
The expected value of $\overline\mu_{2m}$ on permutations of size $n$ is given by the formula
\[
    \mathbb{E}(\overline\mu_{2m})
    = \frac{1}{m+1} \binom{n+m}{2m+1}.
\]
\end{proposition}

\begin{proof}
We will make use of the following identity of binomial coefficients: for $\alpha\geq0$ and $\beta\geq 1$, we have
\begin{align}
\label{eq:a+b+1 choose a+2 identity}
    \sum_{\ell=0}^{\beta-1} (\beta-\ell) \binom{\alpha+\ell}{\alpha}
    &= \binom{\alpha+\beta+1}{\alpha+2}.
\end{align}
A combinatorial proof of this identity is as follows:
the right-hand side counts the $(\alpha+2)$-element subsets of $[\alpha+\beta+1]$.
The left-hand side does the same, where the $\ell$th term counts those $(\alpha+2)$-element subsets of $[\alpha+\beta+1]$ whose second largest element is $\alpha+1+\ell$.

We now proceed with the calculation. 
\begin{align*}
    \mathbb E\left(\overline\mu_{2m}\right) 
    &= \sum_{i=1}^n \mathbb E\left(\binom{w(i)-i+m-1}{2m}\right) \\
    &= \sum_{i=1}^n \sum_{j=1}^n \frac1n \binom{j-i+m-1}{2m} \\
    &= \frac1n \sum_{k=1-n}^{n-1} (n-|k|) \binom{k+m-1}{2m} \\
    &= \frac1n\left( \sum_{k=m+1}^{n-1} (n-k) \binom{k+m-1}{2m}
    + \sum_{k=m}^{n-1} (n-k) \underbrace{\binom{-k+m-1}{2m}}_{(-1)^{2m}\binom{k+m}{2m}}\right) \\
    &= \frac1n \left(
    \underbrace{
        \sum_{\ell=0}^{n-m-2} (n-m-1-\ell) \binom{2m+\ell}{2m}
    }_{\text{
        Equation~\eqref{eq:a+b+1 choose a+2 identity} with
        $\alpha=2m$, 
        $\beta=n-m-1$
    }}
    + \underbrace{
        \sum_{\ell=0}^{n-m-1} (n-m-\ell) \binom{2m+\ell}{2m}
    }_{\text{
        Equation~\eqref{eq:a+b+1 choose a+2 identity} with
        $\alpha=2m$,
        $\beta=n-m$
    }}
    \right) \\
    &= \frac1n\left( \binom{n+m}{2m+2} + \binom{n+m+1}{2m+2}\right) \\
    &= \frac1n\left(\frac{n-m-1}{2m+2} + \frac{n+m+1}{2m+2}\right) \binom{n+m}{2m+1}\\
    &= \frac{1}{m+1} \binom{n+m}{2m+1}.\qedhere
\end{align*}
\end{proof}

\section{Pattern-positive expansions}\label{section:reduced words}

We now turn our attention from pattern-finite expansions to another enticing class of expansions: pattern-positive statistics. Recall that a statistic is \textit{pattern-positive} if the coefficients in its pattern expansion are nonnegative. We start with a simple observation about pattern-positive statistics.

\begin{lemma}\label{lem:positivemonotonic}
    Pattern-positive statistics are nonnegative and monotonic with respect to pattern containment. In other words, if $\sigma$ is a pattern-positive statistic then $\sigma(w)\geq 0$ for all $w$, and whenever $p$ occurs as a pattern in $w$, we have $\sigma(p)\leq \sigma(w).$
\end{lemma}

\begin{proof}
    For any pattern-positive statistic $\sigma$, the expression $\stat(w) = \sum_q \coeff{\stat}{q} [q](w)$ is nonnegative because both $\coeff{\stat}{q}$ and $[q](w)$ are always nonnegative. For $p$ occurring as a pattern in $w$ we have $[q](p)\leq [q](w)$ for all $q$, as any occurrence of $q$ in $p$ also gives an occurrence of $q$ in $w$. We therefore have
    \[ \stat(p) = \sum_{q} \stat_q [q](p)\leq \sum_{q} \stat_q [q](w) = \stat(w).\]
\end{proof}

By Lemma~\ref{lem:positivemonotonic}, any integer-valued pattern-positive statistic $\stat$ has nonnegative integer outputs. There is therefore no loss of generality in studying the pattern-positivity of statistics $\stat$ whose outputs $\stat(w)$ are described as the cardinalities of finite sets. This motivates the following definition. An \textit{enumerative statistic} $\stat$ is a statistic given by the cardinality $\stat(w) = |\set(w)|$ of some finite set $\set(w)$ associated to $w$. Note that any statistic with nonnegative integer values can be given as an enumerative statistic, but we use the term so we may reason about the associated sets of objects directly. Many common statistics in algebraic combinatorics are defined enumeratively. For example, the statistics counting reduced words, descents, and length can all be defined in this way. Some common statistics, for example the Schubert specialization statistic, have several natural enumerative descriptions. Our results below do not depend on the particular enumerative interpretation chosen for the statistic.

Monotonicity, with respect to pattern containment, of an enumerative statistic has a clear combinatorial meaning. Indeed, an enumerative statistic $\stat(w) = |\set(w)|$ is monotonic if and only if there are injective maps $\set(p)\hookrightarrow \set(w)$ for any $p$ that occurs as a pattern in $w$. 

Our motivating question is: What kind of combinatorial structure on the sets $\set(w)$ corresponds to pattern-positivity of the corresponding enumerative statistic? By Lemma~\ref{lem:positivemonotonic} pattern-positive statistics are, in particular, monotonic, so one might expect the structure to consist of inclusions $\set(p)\hookrightarrow \set(w)$ with some additional properties. As we will show, the following combinatorial structure is one possible answer to this question.

\begin{definition}\label{defn:incexcstructure}
An \textit{inclusion-exclusion structure} for an enumerative statistic $|\set(w)|$ is a family of inclusions $\set(p)\hookrightarrow \set(w)$, indexed by pattern occurrences $p$ in $w$, that respects intersections. That is, for $I,J\subseteq [1,n]$, we have
\[ \set(w|_{I\cap J}) = \set(w|_I) \cap \set(w|_J), \]
where we identify these sets with their images in $\set(w)$. 
For an enumerative statistic $\sigma$ equipped with an inclusion-exclusion structure, an object $a\in \set(w)$ is \textit{essential} if it is not in the image of the map $\set(p)\hookrightarrow \set(w)$ for any proper pattern $p$ of $w$ (that is, the pattern $p$ is neither $\varnothing$ nor $w$). The set of essential elements of $\set(w)$ is the \textit{essential set}, denoted $\Eset(w)\subseteq \set(w)$.
\end{definition}

\begin{proposition}\label{prop:positivitycriterion}
    Any enumerative statistic $\stat(w) = |\set(w)|$ with an inclusion-exclusion structure is pattern-positive with coefficients $\coeff{\stat}{p} = |\Eset(p)|$.
\end{proposition}

\begin{proof}
    By Lemma~\ref{lem:coeff in terms of stat}, we have
    \[ \coeff{\stat}{p} = \sum_{I\subseteq [1,n]} (-1)^{n-|I|}\cdot\stat(p|_I) = \sum_{I\subseteq [1,n]} (-1)^{n-|I|}\cdot|\set(p|_I)|. \]
    
    For $i=1,\dots,n$, define $A_i= \set(p|_{[1,n]\setminus\{i\}})$ considered as a subset of $\set(p)$ via the maps defining the inclusion-exclusion structure. The definition of an inclusion-exclusion structure implies that for any $I\subseteq[1,n]$, we have $\set(p|_{I}) = \bigcap_{i\in I} A_i$. The claim now follows from the inclusion-exclusion formula,
    \[
        \coeff{\stat}{p} = \sum_{I\subseteq [1,n]} (-1)^{n-|I|}\cdot\left|\bigcap_{i\in I} A_i\right| = \left|\set(p)\setminus \bigcup_{i=1}^n A_i\right| = |\Eset(p)|.
        \qedhere
    \]
\end{proof}

In the following section we give a nontrivial example of an inclusion-exclusion structure on the sets of reduced words of a permutation. In the following example we illustrate the definitions above for a simple statistic that has a readily available and positive pattern expansion.

\begin{example}
    Let $\sigma(w)$ be the statistic enumerating $\IS(w)$, the set of increasing (possibly empty) subsequences of $w$. For example, we have $\IS(132) = \{ \varnothing, 1,3,2,13,12\}$ and so $\sigma(132) = 6$. We clearly have the expansion 
    \[ \sigma = \sum_{n\geq 0} [\mathrm{id}_n] = 1 + [1] + [12]+[123]+[1234]+\cdots.\]
    For any pattern occurrence of $p$ in $w$, any increasing subsequence of $p$ naturally corresponds to an increasing subsequence in $w$ consisting of the corresponding entries. This gives maps $\IS(p)\hookrightarrow \IS(w)$. One easily checks that these maps satisfy the definition of an inclusion-exclusion structure. 
    An increasing subsequence of $w$ on indices $J$ is in the image of the map $\set(w|_J)\hookrightarrow\set(w)$. Therefore, the only essential increasing subsequences are the full subsequences $12\cdots n$ in the identity permutations $\mathrm{id}_n = 12\cdots n\in S_n$. Proposition~\ref{prop:positivitycriterion} then recovers the pattern expansion for $\sigma$ given above.
\end{example}

In fact, an inclusion-exclusion structure on an enumerative statistic $\sigma$ gives a combinatorial interpretation for every term in the expression for $\sigma(w)$ obtained by pattern expanding $\sigma$. Every element of $\set(w)$ is, in a unique way, the image of an essential element $\Eset(p)$ of some pattern $p$ of $w$. For a given $p$ and $w$, the number of such elements of $\set(w)$ is equal to $|\Eset(p)|\cdot [p](w)$

\subsection{A case study: enumerating reduced words}

A \textit{reduced word} for a permutation $w$ is an expression $w=s_{i_\ell}\cdots s_{i_1}$ as a product of adjacent transpositions, where $s_i$ transposes $i$ and $i+1$ and fix all other indices. Let $\RW(w)$ be the set of all reduced words for $w$ and $\rw(w) = |\RW(w)|$. Using the techniques described in Section~\ref{section:tools}, we compute the first part of the pattern expansion of $\rw$ to be
\begin{align*}
    \rw = [\varnothing] + [321] + [2143] + [2413] + [2431] + [3142] + [3241] + [3412]\\ + 2[3421] + [4132] + [4213] + 3[4231] + 2[4312] + 11[4321] + \cdots.
\end{align*}
In this section we construct an inclusion-exclusion structure on the reduced words of permutations, establishing the pattern-positivity of the statistic $\rw(w)$.

An \textit{inversion} of $w$ indicates an occurrence of the pattern $21$: a pair of indices $(a,b)$ such that $a<b$ and $w(a)>w(b)$. We say that a transposition $s_{i_k}$ in a reduced word $w=s_{i_\ell}\cdots s_{i_1}$ \textit{creates an inversion $(a,b)$} if $(a,b)$ is an inversion of $s_{i_k}\cdots s_{i_1}$ but is not an inversion of $s_{i_{k-1}}\cdots s_{i_1}$. Explicitly, this means that $s_{i_{k-1}}\cdots s_{i_1}(a) = i_k$ and $s_{i_{k-1}}\cdots s_{i_1}(b) = i_k+1$. 

We read reduced words from right to left (as function compositions) and will sometimes represent reduced words as \textit{wiring diagrams} (see Figure~\ref{fig: rw projection}) that record the order of the inversions as wire crossings from right to left. Our convention for wiring diagrams is that the ground set $1,\dots,n$ on which our permutations act is identified with the wires on the right hand side from top to bottom. 

For a pattern $p=w|_A$ occurring on index set $I \subseteq [1,n]$, there is a natural restriction map $\pi:\RW(w)\to \RW(p)$ given by recording only the inversions between indices in $I$ in the reduced word for $w$. In terms of wiring diagrams, this operation is given by removing all wires starting, on the right, at indices outside of $I$. 

\begin{example}\label{ex:restriction}
    The permutation $w=3421$ contains the pattern $p=231$ on the indices $I = \{1,2,4\} = \{i_1,i_2,i_3\}$. The corresponding restriction of the reduced word $s_2 s_1 s_2 s_3 s_2\in \RW(3421)$ to $\RW(231)$ can be found by removing the wire corresponding to $w(3)=2$, position $3$ on the right and position $2$ on the left, from the wiring diagram as shown in Figure~\ref{fig: rw projection}.
    \begin{figure}[htbp]
        \begin{center}
            \begin{tikzpicture}
            \pic[
            rotate=-90,
            braid/.cd,
            every strand/.style={thick},
            strand 1/.style={black},
            strand 2/.style={black},
            strand 3/.style={dashed},
            strand 4/.style={black},
            gap = 0.01] (coords)
            {braid={s_2 s_3 s_2 s_1 s_2 }};
            \node[at=(coords-1-s),pin=east:{$1=i_1$}] {};
            \node[at=(coords-2-s),pin=east:{$2 = i_2$}] {};
            \node[at=(coords-3-s),pin=east:3] {};
            \node[at=(coords-4-s),pin=east:{$4=i_3$}] {};

            \node at (3.2,-1.5) {$\xmapsto{\displaystyle \quad \pi \quad}$};
            
            \pic at (7,-.5) [
            rotate=-90,
            braid/.cd,
            every strand/.style={thick},
            strand 1/.style={black},
            strand 2/.style={black},
            strand 3/.style={black},
            gap = 0.01] (coords)
            {braid={s_2 s_1 }}; 
            \node[at=(coords-1-s),pin=east:1] {};
            \node[at=(coords-2-s),pin=east:2] {};
            \node[at=(coords-3-s),pin=east:3] {};
            \end{tikzpicture}
        \end{center}
        \caption{The reduced word $s_2 s_1 s_2 s_3 s_2\in \RW(3421)$ restricts to $s_1s_2\in \RW(231)$ for the $231$-pattern that occurs on indices $\{1,2,4\}$ in $3421$.}
        \label{fig: rw projection}
    \end{figure}
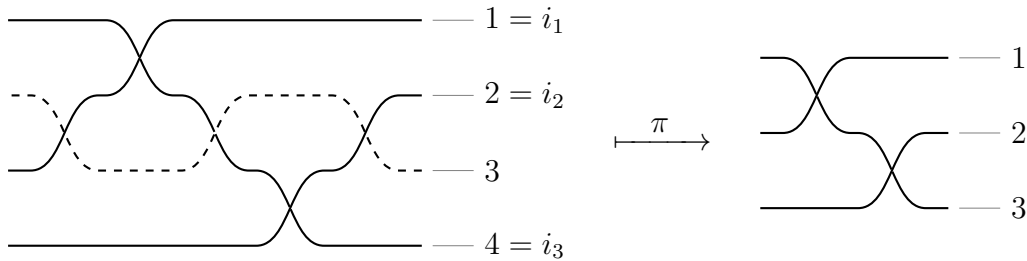
\end{example}

One can check that these restriction maps $\pi:\RW(w)\to\RW(p)$ are surjective; i.e., every reduced word for $p$ is the restriction of some reduced word for $w$. Our inclusion-exclusion system on reduced words defined below consists of specially chosen right-inverses $\iota:\RW(p)\hookrightarrow \RW(w)$ to these restriction maps.

\begin{definition}
    For a pattern occurrence $p$ in $w$, the \textit{minimal lift} of a reduced word $x\in \RW(p)$ to $w$ is the reduced word $\iota(x) = s_{i_\ell}\cdots s_{i_1}$ of $w$ whose sequence of indices $(i_1,\dots,i_\ell)$ is the lexicographic minimum among $\pi^{-1}(x)$. In other words, we have $\iota(x)\in \pi^{-1}(x)$ and for any other word $s_{i_\ell'}\cdots s_{i_1'}\in \pi^{-1}(x)$, there is an index $k$ such that $i_1 = i_1', \, \cdots, i_k = i_k'$, and $i_{k+1}<i_{k+1}'$. 
\end{definition}

The following terminology will be useful in describing the minimal lift map more concretely.

\begin{definition}
    We say that a word $s_{i_k}\cdots s_{i_1}$ is a \textit{partial reduced word} for a permutation $w$ if it is a right-side prefix of a reduced word for $w$. Equivalently, $s_{i_k}\cdots s_{i_1}$ is a partial reduced word for $w$ if \[ \ell(w(s_{i_k}\cdots s_{i_1})^{-1}) = \ell(w s_{i_1}\cdots s_{i_k}) = \ell(w)-k.\]   A transposition $s_i$ is called an \textit{extension} of a partial reduced word $s_{i_k}\cdots s_{i_1}$ for $w$ if $s_is_{i_k}\cdots s_{i_1}$ is also a partial reduced word for $w$, i.e. if $\ell(w s_1\cdots s_{i_k}s_i)=\ell(w)-k-1$.
\end{definition}

The minimal lift of a reduced word $x\in \RW(p)= \RW(w|_I)$ to $w$ can be constructed inductively. Indeed, suppose that we have constructed some initial (right) segment $s_{i_k}\cdots s_{i_1}$ of $\iota(x)$, so that $\pi(y)$ agrees with some initial segment of the given reduced word $x$ for $p$. The next transposition $s_{i_{k+1}}$ in $\iota(x)$ is defined by setting $ i_{k+1}$ to be the minimal index $i$ such that $s_i s_{i_k}\cdots s_{i_1}$ is a partial reduced word for $w$ and $\pi(s_i s_{i_k}\cdots s_{i_1})$ is some initial segment of the given reduced word $x\in \RW(p)$.

\begin{example}\label{ex:lift}
    Consider the permutation $w = 4231$ containing the pattern $p=321$ on the set of indices $I = \{1,3,4\}$. We will construct the minimal lift of the reduced word $x = s_2 s_1 s_2$ for $p$ to $w$. 
    \begin{enumerate}
        \item The extensions of the empty partial reduced word for $w$ are $s_1$ and $s_3$. These transpositions restrict to the empty word and $s_2$ respectively, both of which agree with some initial (right) segment of the given word $x= s_2 s_1 s_2$, so we set $s_{i_1}=s_1$ because it has the smaller index.
        \item The extensions of the partial reduced word $s_1$ for $w$ are $s_2$ and $s_3$. The partial reduced word $s_2 s_1$ restricts to $s_1$ as a partial reduced word for $p$, while $s_3 s_1$ restricts to $s_2$. Only the latter restriction is an initial segment of $x=s_2 s_1 s_2$, so we choose $s_{i_2} = s_3$.
        \item There is only one extension of the partial reduced word $s_3 s_1$ for $w$, namely $s_2$. The restriction of $s_2 s_3 s_1$ to $p$ is $s_1 s_2$, which is an initial segment of $x$, so we take $s_{i_3} = s_2$.
        \item The extensions of the partial reduced word $s_2 s_3 s_1$ for $w$ are $s_1$ and $s_3$. The partial reduced word $s_1 s_2 s_3 s_1$ restricts to $s_1 s_2$ while the partial reduced word $s_3 s_2 s_3 s_1$ restricts to $s_2 s_1 s_2$. Both are initial segments of $x$, so we choose $s_{i_4} = s_1$ because it has the smaller index.
        \item There is only one extension of the partial reduced word $s_1 s_2 s_3 s_1$ for $w$, namely $s_3$. The restriction of $s_3 s_1 s_2 s_3 s_1$ to $p$ is $x$, so we take $s_{i_5} = s_3$.
    \end{enumerate}
    Thus, the minimal lift of $x$ to $w$ is the reduced word $s_3 s_1 s_2 s_3 s_1$. These reduced words are shown as wiring diagrams in Figure~\ref{fig: wiring diagram lift}. 
    
    \begin{figure}[htbp]
        \begin{center}
            \begin{tikzpicture}
            \pic at (0,-.5) [
            rotate=-90,
            braid/.cd,
            every strand/.style={thick},
            strand 1/.style={black},
            strand 2/.style={black},
            strand 3/.style={black},
            gap = 0.01] (coords)
            {braid={s_2 s_1 s_2}}; 
            \node[at=(coords-1-s),pin=east:1] {};
            \node[at=(coords-2-s),pin=east:2] {};
            \node[at=(coords-3-s),pin=east:3] {};

            \node at (2.5,-1.5) {$\xmapsto{\displaystyle \quad \iota \quad}$};
            
            \pic at (9,0) [
            rotate=-90,
            braid/.cd,
            every strand/.style={thick},
            strand 1/.style={black},
            strand 2/.style={dashed},
            strand 3/.style={black},
            strand 4/.style={black},
            gap = 0.01] (coords)
            {braid={s_1 s_3 s_2 s_1 s_3 }};
            \node[at=(coords-1-s),pin=east:{$1=a_1$}] {};
            \node[at=(coords-2-s),pin=east:2] {};
            \node[at=(coords-3-s),pin=east:{$3 = a_2$}] {};
            \node[at=(coords-4-s),pin=east:{$4=a_3$}] {};
            \end{tikzpicture}
        \end{center}
        \caption{The minimal lift of the reduced word $s_2s_1s_2\in \RW(321)$ from the $321$-pattern that occurs on indices $\{1,3,4\}$ in $4231$ is $s_3 s_1 s_2 s_3 s_1\in \RW(4231)$.}
        \label{fig: wiring diagram lift}
    \end{figure}
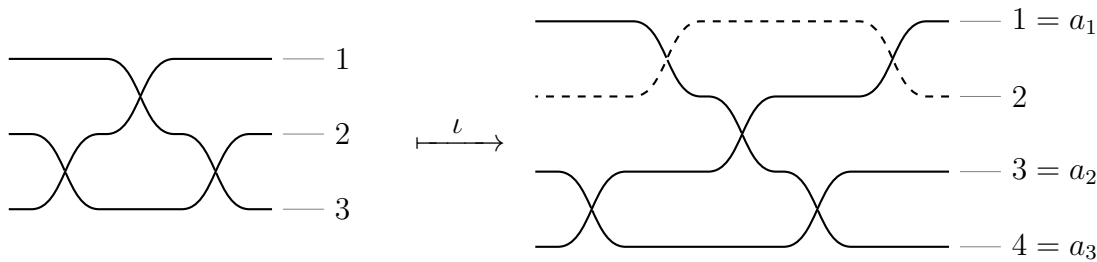
\end{example}

\begin{proposition}\label{prop:composition}
    The composition of two minimal lifts is a minimal lift. In other words, for any pattern $p$ in $w$ and pattern $q$ in $p$, the following diagram of minimal lift maps commutes.
    \[\begin{tikzcd}
        & \RW(w) \\
        \RW(q) \arrow[r] \arrow[ur] & \RW(p) \arrow[u]
    \end{tikzcd}\]
\end{proposition}

\begin{proof}
    Let $x\in \RW(q)$ be a reduced word with minimal lifts $y\in \RW(p)$ and $z\in \RW(w)$. We claim that it suffices to show that $z$ restricts to $y\in \RW(p)$. Indeed, this $z$ is the minimal reduced word for $w$ that restricts to $x\in \RW(q)$, which includes all reduced words that restrict to $y\in \RW(p)$. Therefore if $z$ restricts to $y\in \RW(p)$, it must be the minimal such word.

    Now we will show that the restriction of $z$ to $\RW(p)$ is $y$, by induction on the length of their initial (right) segments that agree. We use the notation $I\subseteq J \subseteq [1,n]$ for the index sets with $q=w|_I$ and $p=w|_J$. Suppose $s_{i_k}\cdots s_{i_1}$ is some initial segment of $z$ such that the restriction of $s_{i_{k-1}}\cdots s_{i_1}$ to $p$ coincides with some initial segment of $y$, and that the next transposition $s_{i_k}$ in $z$ creates an inversion between the pair of distinct indices $j,j'\in J$. By the minimality of $z$ as a lift of $x$, any extension $s_h$ of this partial word with $h<i_k$ must create an inversion between indices $i,i'\in I$ that is not the next inversion created in the given reduced word $x$. 
    
    Now consider the restriction of $s_{i_{k-1}}\cdots s_{i_1}$ to a partial reduced word for $p$ and its extensions. The restriction of any extension of $s_{i_{k-1}}\cdots s_{i_1}$ that creates an inversion between indices in $J$ will again be an extension of the restricted partial word.
    We claim that there are no additional extensions of the restricted word with smaller index than the restriction of $s_{i_k}$.
    Indeed, such extensions correspond to inversions of $w$ between indices $j_1,j_2 \in J$ which, after applying the partial word $s_{i_{k-1}}\cdots s_{i_1}$ are separated only by indices not in $J$.
    But observe that there must then be some extension of $s_{i_{k-1}}\cdots s_{i_1}$ creating an inversion between some pair taken from $j_1$, $j_2$, and these additional indices, which will either be $(j_1,j_2)$ itself or involve some index not in $J$. Since the only extensions of $s_{i_{k-1}}\cdots s_{i_1}$ with index smaller than $i_k$ are between indices in $I\subseteq J$, there can be no such additional extensions of the restriction of $s_{i_{k-1}}\cdots s_{i_1}$ with index smaller than the restriction of $s_{i_k}$. It follows that the restriction of $s_{i_k}$ is the next transposition in $y$, which completes the proof by induction.
\end{proof}

The main result of this section is that minimal lifts form an inclusion-exclusion structure.

\begin{theorem}\label{thm: reduced word structure}
    Minimal lift maps define an inclusion-exclusion structure on reduced words.
\end{theorem}

\begin{proof}
    By Definition~\ref{defn:incexcstructure}, we need to show that given $w\in S_n$ and index sets $I,J\subseteq [1,n]$, we have
    \[ \iota_{I\cap J}(\RW(w|_{I\cap J})) = \iota_I(\RW(w|_I))\cap \iota_J(\RW(w|_J)), \]
    where $\iota_A$ is the minimal lift from $\RW(w|_A)\hookrightarrow \RW(w)$. We proceed by checking both containments.

    On one hand, Proposition~\ref{prop:composition} implies that the composition of minimal lifts
    \[ \RW(w|_{I\cap J})\hookrightarrow \RW(w|_I)\hookrightarrow \RW(w) \]
    coincides with $\iota_{I\cap J}$. This implies that $\iota_{I\cap J}(\RW(w|_{I\cap J}))\subseteq \iota_I(\RW(w|_I))$. The same argument shows that $\iota_{I\cap J}(\RW(w|_{I\cap J}))\subseteq \iota_J(\RW(w|_J))$ as well, which establishes one of the desired containments.
    
    For the other direction, let $x = s_{i_\ell}\cdots s_{i_1}$ be a reduced word for $w$ in $\iota_I(\RW(w|_I))\cap \iota_J(\RW(w|_J))$. We write $x|_I$, $x|_J$, and $x|_{I\cap J}$ for the restrictions of $x$ to reduced words on the patterns of $w$ given by the respective sets of indices. Note that we must have $x = \iota_I(x|_I)$ and $x=\iota_J(x|_J)$. Also, the further restrictions of $x|_I$ and $x|_J$ to reduced words for $w|_{I\cap J}$ both coincide with $x|_{I\cap J}$ because the order of the inversions among pairs of indices in $I\cap J$ are both induced from $x$. 
    
    We will show that $x$ is the minimal lift $\iota_{I\cap J}(x|_{I\cap J})$ by induction on the length of their initial (right) segments that agree. Indeed, that some (possibly empty) initial segment $s_{i_k}\cdots s_{i_1}$ of $x$ coincides with the initial segment of $\iota_{I\cap J}(x|_{I\cap J})$ of the same length. The extensions of this partial reduced word for $w$ are those letters $s_i$ for which we have $\ell(ws_{i_1}\cdots s_{i_k}s_i)<\ell(ws_{i_1}\cdots s_{i_k})$. Since $x$ is the minimal lift of $x|_I$, the index $i_{k+1}$ of the next transposition in $x$ is the smallest index among extensions (as a partial word of $w$) that either creates an inversion involving an index outside of $I$ or restricts to the next inversion in the reduced word $x|_I$. Similarly, because we also have $x=x|_J$, $i_{k+1}$ is also the smallest index among extensions that either create an inversion involving an index outside of $J$ or restrict the next inversion in the reduced word $x|_J$. Combining these conditions, the index $i_{k+1}$ must be the smallest index among extensions that either create an inversion involving an index in $I\cap J$ or give the next inversion in $x|_{I\cap J}$. This is, by definition, the index of the next inversion in $\iota_{I\cap J}(x|_{I\cap J})$, so we may conclude by induction that $x = \iota_{I\cap J}(x|_{I\cap J})$. This completes the proof.
\end{proof}

As in the previous section, we say that a reduced word for $w$ is an \textit{essential reduced word} if it is not in the image of any minimal lift map $\RW(p)\hookrightarrow \RW(w)$ for a proper pattern $p$ in $w$. Combining Proposition~\ref{prop:positivitycriterion} with Theorem~\ref{thm: reduced word structure} yields the following result.

\begin{corollary}
    The number of reduced words is a pattern-positive statistic, and the coefficient on $[p]$ in the pattern expansion is the number of essential reduced words for $p$.
\end{corollary}

In fact, we have an even more refined combinatorial interpretation for this pattern expansion. For each reduced word $x\in RW(w)$, define the \textit{essential index set} $I\subseteq [1,n]$ to be the set of indices $i$ such that $x$ is \textit{not} in the image $\RW(w|_{\{1,\dots,\hat{i},\dots,n\}})\xhookrightarrow{\iota} \RW(w)$. The previous results imply that the restriction $y\in \RW(p)$ of $x$ to $p=w|_I$ is an essential reduced word for $p$, and that $x = \iota(y)$. We can use this to obtain a combinatorial interpretation for every term in the expression given by computing $\rw(w)$ via the pattern expansion of $\rw$ as follows.

Partition the set of reduced words $\RW(w)$ into sets $\RW_I(w)$ indexed by $I\subseteq [1,n]$ consisting of those reduced words whose essential index set is $I$. In the pattern expansion $\rw(w) = \cdots+ \coeff{\rw}{p}[p](w)+\cdots$, each pattern occurrence $p = w|_{I}$ contributes $\coeff{\rw}{p}$, the number of essential reduced words for $p$, to the resulting expression for $\rw(w)$. We identify this term as corresponding to the subset $\RW_I(w)\subseteq \RW(w)$, giving a complete combinatorial interpretation for the decomposition of $\rw(w)$ given by the pattern expansion of $\rw$.

\begin{example}
    The permutation $w=3421$ has $5$ reduced words, depicted in Figure~\ref{fig:five wiring diagrams} as wiring diagrams with the dashed wires corresponding non-essential indices.
    These correspond to the terms in the expression
    \begin{align*}
        \rw(3421) & = (1+[321]+2[3421]+\cdots)(3421) \\
        & = 1 + [321](3421) + 2 [3421](3421)\\
        & = 1 + 2 + 2 = 5.
    \end{align*} 

    For example, $s_2 s_1 s_2$ is the unique essential reduced word for the permutation $p=321$. Each of the two occurrences of $321$ as a pattern in $w = 3421$ induces one reduced word for $w$, the minimal lift of $s_2 s_1 s_2$ from the $p$-pattern to $w$. 
    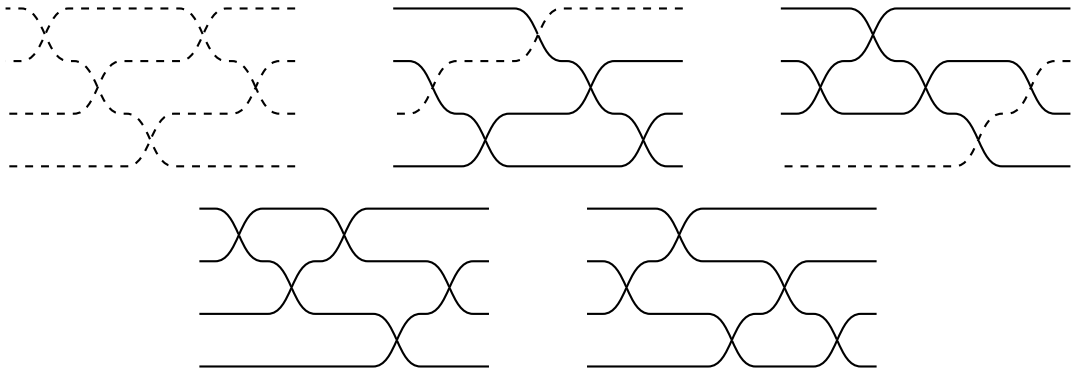
\begin{figure}[htbp]
        \centering
        \begin{tikzpicture}
            \pic[
            rotate=-90,
            braid/.cd,
            every strand/.style={thick},
            strand 1/.style={dashed},
            strand 2/.style={dashed},
            strand 3/.style={dashed},
            strand 4/.style={dashed},
            gap = 0.01,
            scale = .7] (coords)
            {braid={s_2 s_1 s_3 s_2 s_1 }};
        \end{tikzpicture}
                \hspace{1cm}
        \begin{tikzpicture}
            \pic[
            rotate=-90,
            braid/.cd,
            every strand/.style={thick},
            strand 1/.style={dashed},
            strand 2/.style={black},
            strand 3/.style={black},
            strand 4/.style={black},
            gap = 0.01,
            scale = .7] (coords)
            {braid={s_3 s_2 s_1 s_3 s_2 }};
        \end{tikzpicture}
                \hspace{1cm}
        \begin{tikzpicture}
            \pic[
            rotate=-90,
            braid/.cd,
            every strand/.style={thick},
            strand 1/.style={black},
            strand 2/.style={dashed},
            strand 3/.style={black},
            strand 4/.style={black},
            gap = 0.01,
            scale = .7] (coords)
            {braid={s_2 s_3 s_2 s_1 s_2 }};
        \end{tikzpicture}
        
        \vspace{.5cm}
        
        \begin{tikzpicture}
            \pic[
            rotate=-90,
            braid/.cd,
            every strand/.style={thick},
            strand 1/.style={black},
            strand 2/.style={black},
            strand 3/.style={black},
            strand 4/.style={black},
            gap = 0.01,
            scale = .7] (coords)
            {braid={s_2 s_3 s_1 s_2 s_1 }};
        \end{tikzpicture}
                \hspace{1cm}
        \begin{tikzpicture}
            \pic[
            rotate=-90,
            braid/.cd,
            every strand/.style={thick},
            strand 1/.style={black},
            strand 2/.style={black},
            strand 3/.style={black},
            strand 4/.style={black},
            gap = 0.01,
            scale = .7] (coords)
            {braid={s_3 s_2 s_3 s_1 s_2 }};
        \end{tikzpicture}
        \caption{The five wiring diagrams for the permutation $3421$ with dashed wires corresponding to nonessential indices, corresponding to the expansion of $\rw(3421)$ from the pattern expansion of $\rw$.}
        \label{fig:five wiring diagrams}
    \end{figure}
\end{example}

\section{Further directions}\label{section:future}

This work opens up many questions about permutation statistics, the tools that are available for studying them, the utility of permutation patterns as a mechanism for studying other topics (i.e., not just as interesting phenomena for their own sake), and the potential impact of those patterns on broader properties. While we could not possible list all follow-up questions in these, and other, avenues, we suggest a few of them here.

What are the pattern expansions of other well-studied permutation statistics listed in \cite{findstat}? For which statistics in the literature are the expansions pattern-finite or pattern-positive? Which are both? How can calculations that are easy to perform in the context of a pattern expansion, such as the expected value of a statistic, shed light on properties of the statistic itself? 
How can different statistics, in the common language of pattern expansions, be held up against each other (\`a la \cite{BermanTenner}), and what questions can such a comparison answer? What additional properties are shared by pattern-finite (or pattern-positive) expansions? Similarly, what other classes of pattern expansions have useful properties and applications? Additionally a statistic $\stat$ induces a second-order statistic, namely $p \mapsto \coeff{\stat}{p}$. What properties does this iterative process have?

Finally, it would be interesting to examine if there is any connection between the pattern-positivity of the reduced word statistic and the conjectural pattern-positivity of the Schubert specialization statistic.
One might further wonder whether other related statistics also exhibit this pattern-positivity.
For instance, computational evidence suggests that for each $k\in \N$, the permutation statistic $\mathsf{hw}_k$ whose value on $w$ is the number of \textit{Hecke words} (or \textit{Demazure words}) of $w$ of length $\ell(w) + k$ is pattern-positive (see e.g. \cite{Weigandt21} for a definition of Hecke word).
We remark that $\mathsf{hw}_0 = \rw$, so this would generalize the pattern-positivity of the reduced word statistic.

\nocite{*}
\bibliographystyle{plain}
\bibliography{peps.bib}

@article{BabsonSteingrimsson,
  title={Generalized permutation patterns and a classification of the Mahonian statistics.},
  author={Babson, Eric and Steingr{\'\i}msson, Einar},
  journal={S{\'e}minaire Lotharingien de Combinatoire [electronic only]},
  volume={44},
  pages={B44b--18},
  year={2000},
  publisher={Universit{\"a}t Wien, Fakult{\"a}t f{\"u}r Mathematik}
}

@article{BermanTenner,
  title={Pattern-Functions, Statistics, and Shallow Permutations},
  author={Yosef Berman and Bridget Eileen Tenner},
  journal={Electronic Journal of Combinatorics},
  year={2021},
  volume={29},
  url={https://api.semanticscholar.org/CorpusID:239049715}
}

@article{BrandenClaesson,
  title={Mesh Patterns and the Expansion of Permutation Statistics as Sums of Permutation Patterns},
  author={Petter Br{\"a}nd{\'e}n and Anders Claesson},
  journal={Electronic Journal of Combinatorics},
  year={2011},
  volume={18},
  url={https://api.semanticscholar.org/CorpusID:18533708}
}

@article{Dennin,
     author = {Dennin, Hugh},
     title = {Pattern bounds for principal specializations of $\beta ${-Grothendieck} polynomials},
     journal = {Algebraic Combinatorics},
     pages = {745--763},
     year = {2025},
     publisher = {The Combinatorics Consortium},
     volume = {8},
     number = {3},
     doi = {10.5802/alco.422},
     language = {en},
     url = {https://alco.centre-mersenne.org/articles/10.5802/alco.422/}
}

@misc{FindStat,
  author={Martin Rubey and Christian Stump and others},
  title={{FindStat} - {T}he combinatorial statistics database},
  howpublished={\url{http://www.FindStat.org}},
  url={http://www.FindStat.org},
}

@article{Gao,
  title={Principal specializations of Schubert polynomials and pattern containment},
  author={Gao, Yibo},
  journal={European Journal of Combinatorics},
  volume={94},
  pages={103291},
  year={2021},
  publisher={Elsevier}
}

@article{MeszarosTanjaya,
     author = {M\'esz\'aros, Karola and Tanjaya, Arthur},
     title = {Inclusion-exclusion on {Schubert} polynomials},
     journal = {Algebraic Combinatorics},
     pages = {209--226},
     year = {2022},
     publisher = {The Combinatorics Consortium},
     volume = {5},
     number = {2},
     doi = {10.5802/alco.200},
     language = {en},
     url = {https://alco.centre-mersenne.org/articles/10.5802/alco.200/}
}

@article{Stanley,
title = {On the Number of Reduced Decompositions of Elements of Coxeter Groups},
journal = {European Journal of Combinatorics},
volume = {5},
number = {4},
pages = {359-372},
year = {1984},
issn = {0195-6698},
doi = {https://doi.org/10.1016/S0195-6698(84)80039-6},
url = {https://www.sciencedirect.com/science/article/pii/S0195669884800396},
author = {Richard P. Stanley}
}

@article{Tenner-prisms,
  title={Prism permutations in the {B}ruhat order},
  author={Tenner, Bridget Eileen},
  journal={Advances in Applied Mathematics},
  volume={159},
  pages={102734},
  year={2024}
}

@article{Vargas,
  title={Hopf algebra of permutation pattern functions},
  author={Vargas, Yannic},
  journal={Discrete Mathematics \& Theoretical Computer Science, Proceedings of the 26th International Conference on Formal Power Series and Algebraic Combinatorics},
  pages={839-850},
  year={2014},
  publisher={Episciences. org}
}

@article{Weigandt,
    author = {Weigandt, Anna E.},
    title = {Schubert polynomials, 132-patterns, and Stanley’s conjecture},
    journal = {Algebraic Combinatorics},
    volume = {1},
    number = {4},
    pages = {415-423},
    year = {2018}
}

@article {Weigandt21,
    AUTHOR = {Weigandt, Anna E.},
     TITLE = {Bumpless pipe dreams and alternating sign matrices},
   JOURNAL = {Journal of Combinatorial Theory, Series A},
  FJOURNAL = {Journal of Combinatorial Theory. Series A},
    VOLUME = {182},
      YEAR = {2021},
     PAGES = {Paper No. 105470, 52},
      ISSN = {0097-3165,1096-0899},
   MRCLASS = {05E14 (05E05 14N15 19D99)},
  MRNUMBER = {4258766},
MRREVIEWER = {Edward\ E.\ Allen},
       DOI = {10.1016/j.jcta.2021.105470},
       URL = {https://doi.org/10.1016/j.jcta.2021.105470},
}

\end{document}